\documentclass[12pt]{article}
\usepackage{amsthm}
\usepackage{amsmath}
\usepackage{amsfonts}
\usepackage{amssymb}
\usepackage{amscd}
\usepackage{url}
\usepackage{epsfig}
\usepackage{lscape}
\usepackage{bm}

\newtheorem{thm}{Theorem}
\newtheorem{cor}[thm]{Corollary}
\newtheorem{lem}[thm]{Lemma}
\newtheorem{prop}[thm]{Proposition}
\newtheorem{conj}[thm]{Conjecture}

\theoremstyle{definition}

\newtheorem{remark}[thm]{Remark}

\newcommand{\C}{\mathbb{C}}
\newcommand{\Z}{\mathbb{Z}}
\newcommand{\N}{\mathbb{N}}
\newcommand{\Q}{\mathbb{Q}}
\newcommand{\PP}{\mathbb{P}}

\newcommand{\im}{{\rm Im}\,}

\begin{document}
\bibliographystyle{plain}
\title{\bf Polynomial relations among \\ principal minors of a 4${\times}$4-matrix}
\author{Shaowei Lin and Bernd Sturmfels}
\date{}
\maketitle

\begin{abstract}
\noindent 
The image of the principal minor map for $n {\times} n$-matrices is 
shown to be closed.
In the 19th century, Nansen and Muir studied the
implicitization problem of finding all relations among principal minors
when  $n=4$. We complete their
partial results by constructing explicit polynomials
of degree $12$ that scheme-theoretically define
this affine variety and also its projective closure in $\PP^{15}$.
The latter is the
main component in the singular locus of the
$2 {\times} 2 {\times} 2 {\times} 2$-hyperdeterminant.
\end{abstract}

\section{Introduction}
Principal minors of square matrices appear in numerous applications.
A basic question is the Principal Minors Assignment Problem  \cite{HSn}
which asks for necessary and sufficient conditions for a collection
of $2^n $ numbers to arise as the principal minors of an $n {\times} n $-matrix.
When the matrix is symmetric, this question was recently answered by
Oeding \cite{O} who extended work 
of Holtz and Sturmfels \cite{HS} to show that the principal minors of a symmetric matrix 
are characterized by certain hyperdeterminantal equations of degree~four.

This question is harder for general matrices than it is for symmetric matrices. 
For example, consider our Theorem \ref{thm:closure} which says the image 
of the principal minor map is closed. The same statement is
trivially true for symmetric matrices, but the proof for non-symmetric matrices is quite
subtle.

We denote the principal minors of a complex $n {\times} n$-matrix $A$ by $A_I$ where $I \subseteq [n] = \{1, 2, \ldots, n\}$. Here, $A_I$ is the minor of $A$ whose rows and columns are indexed by $I$, including the $0 {\times} 0$-minor $A_\emptyset = 1$. Together, they form a vector $A_*$ of length $2^n$. We are interested in an algebraic characterization of all vectors in $\C^{2^n}$ which can be written in this form. 
The map  $\,\phi_a: \C^{n^2} \rightarrow \C^{2^n}, A \mapsto A_*\,$ is called the
 \emph{affine principal minor map} for $n {\times} n$-matrices.

\begin{thm}
\label{thm:closure}
The image of the affine principal minor map is closed in $\C^{2^n}$. \end{thm}

This result says that $\im \phi_a$ is a complex algebraic variety. 
The dimension of this variety is $n^2-n+1$. This number is an upper
bound because the principal minors remain unchanged under the
transformation  $A \mapsto D A D^{-1}$ for diagonal matrices $D$,
and it is not hard to see that this upper bound is attained \cite{Sto}.
What we are interested in here is the prime ideal of
polynomials that vanish on the irreducible variety  ${\rm Im}\, \phi_a$.
 We determine this prime ideal in the first non-trivial case $n=4$.
Here, we ignore the trivial relation $A_\emptyset = 1$.

\begin{thm}\label{thm:affine}
When $n{=}4$, the prime ideal of the $13$-dimensional variety
 $\,\im \phi_a\,$   is minimally
generated by $65$ polynomials of degree $12$ in the unknowns $A_I$.
\end{thm}

This theorem completes the line of research started by
MacMahon, Muir and Nanson \cite{Ma, M, N} in the late 19th century.
Our proof of Theorem \ref{thm:affine} takes advantage of
their classical results, and it will be presented in Section~3.

Algebraic geometers would consider it more natural to study the projective version of our problem.
We define the \emph{projective principal minor map}~as 
\begin{equation}
\label{projectivephi}
\,\phi \,: \, \C^{n^2} {\times} \C^{n^2} \rightarrow \C^{2^n}, \,
(A,B) \mapsto \bigl( \,{\rm det}(A_I B_{[n] \backslash I}) \,
\bigr)_{I \subseteq [n]} .
\end{equation}
Here we take two unknown $n {\times} n$-matrices $A$ and $B$
to form an $n {\times} 2n$-matrix $(A,B)$, and for each $I \subseteq [n]$,
we evaluate the $n {\times} n$-minor with column indices $I$ on $A$
and  column indices $[n] \backslash I$ on $B$.
The image of $\phi$ is a closed affine cone in $\C^{2^n}$,
to be regarded as a projective variety in 
$\PP^{2^n-1} = \PP(  \C^2 \otimes \cdots \otimes \C^2 )$.

What makes the projective version more interesting than the affine version is that
$\im \phi$ is invariant under the natural action of the group $G = {\rm GL}_2(\C)^n$.
This was observed by JM~Landsberg
(cf.~\cite{HS, O}).  In Section 4 we shall prove

\begin{thm}\label{thm:projective}
When $n=4$, the projective variety
$\im \phi$ is cut out scheme-theoretically by 
$718$ linearly independent homogeneous polynomials of degree $12$.
 This space of polynomials is the direct sum of
$14$ irreducible $G$-modules.
\end{thm}

The polynomials 
described  in Theorems \ref{thm:affine} and \ref{thm:projective}
are available online at
$$ \hbox{\url{http://math.berkeley.edu/~shaowei/minors.html} .}$$

A geometric interpretation of our projective variety is given in Section~5. 
We shall see that $ \im \phi$
is the main component in the singular locus of the
$2 {\times} 2 {\times} 2 {\times} 2$-hyperdeterminant.
This is based on work of  Weyman and Zelevinsky \cite{WZ}, 
and it relates our current study to the computations in \cite{HSYY}.
We begin in Section 2 by rewriting principal minors in terms of
{\em cycle-sums}. The resulting combinatorial structures
 are key to our proof of Theorem~\ref{thm:closure}.

\smallskip

After posting the first version of this paper on the {\tt arXiv},
Eric Rains informed us that some of our findings overlap with
results in Section 4 of his article \cite{BR}
with Alexei Borodin. The relationship to their work,
which we had been unaware of, will be discussed at
the end of this paper, in Remark \ref{rem:BR}.

\section{Cycle-Sums and Closure}

In this section, we define cycle-sums and determine their relationship to the principal minors. We then prove that a certain ring generated by monomials called cycles is integral over the ring generated by the principal minors. This integrality result will be our main tool for proving the closure theorem.

Let $X = (x_{ij})$ be an $n {\times} n$-matrix of indeterminates and $\C[X]$ the polynomial ring generated by these indeterminates. Let $P_I \in \C[X]$ denote the principal minor of $X$ indexed 
by $I \subseteq [n]$, including $P_{\emptyset} = 1$. Together, these minors form a vector $P_*$ of length $2^n$. 
Thus $A_* = P_*(A)$ if $A$ is a complex $n {\times} n$-matrix.
  Now, given a permutation $\pi \in \mathfrak{S}_n$ of $[n]$, define the monomial $c_{\pi} = \prod_{i \neq \pi(i)} x_{i\pi(i)} \in \C[X]$ where the product is taken over the support of $\pi$. We call $c_{\pi}$ a \emph{$k$-cycle} if $\pi$ is a cycle of length $k$. For $I \subseteq [n]$, $|I| \geq 2$, define the \emph{cycle-sum} $C_I = \sum_{\pi \in \mathfrak{C}_I} c_\pi$ where $\mathfrak{C}_I$ is the set of all cycles with support $I$. Also, let $C_{\emptyset} = 1$ and $C_i = x_{ii}$, $i \in [n]$. Together, they form a vector $C_*$ of length $2^n$. The cycles, cycle-sums and principal minors generate subrings $\C[c_*]$, $\C[C_*]$ and $\C[P_*]$ of $\C[X]$. The next result shows that $\C[C_*] = \C[P_*]$.

\begin{prop}
\label{pm-cs formula}
The principal minors and cycle-sums satisfy the following relations: 
for any subset $I \subseteq [n]$ of cardinality $d \geq 1$, we have
\begin{eqnarray}
\label{eq:pmcs1} & P_I  \,\,=\,\, \sum_{I = I_1 \sqcup \ldots \sqcup I_k} (-1)^{k+d}C_{I_1}\cdots C_{I_k} & \\
\label{eq:pmcs2} & C_I \,=\, \sum_{I = I_1 \sqcup \ldots \sqcup I_k} (-1)^{k+d}(k-1)!P_{I_1}\cdots  P_{I_k} &
\end{eqnarray}
where the sums are taken over all set partitions $I_1 \sqcup \ldots \sqcup I_k$ of $I$.
\end{prop}

\begin{proof}
The first equation is Leibniz's formula for the determinant. The second equation is
derived from the first formula by applying the M\"{o}bius inversion formula
\cite[\S 3.7]{Sta} for the lattice of all set partitions of $[n]$.
\end{proof}

Let $\psi:\C^{2^n} \rightarrow \C^{2^n}$ be the polynomial map given by (\ref{eq:pmcs2}). We say 
that $u_* \in \C^{2^n}$ is \emph{realizable as principal minors} if $u_* = P_*(A)$ for some complex matrix $A$. Similarly,  $u_*$ is \emph{realizable as cycle-sums} if $u_* = C_*(A)$ for some~$A$. 

\begin{cor}
\label{co-realizable}
A vector $u_* \in \C^{2^n}$ is realizable as principal minors if and only if 
its image $\psi(u_*)$ is realizable as cycle-sums.
\end{cor}

A monomial in $\C[X]$ can be represented by a directed multigraph on $n$ vertices as follows: label the vertices $1, \ldots, n$ and for each $x_{ij}^k$ appearing in the monomial, draw $k$ directed edges from vertex $i$ to vertex $j$. Cycles $c_{(i_1\ldots i_k)}$ correspond to cycle graphs $i_1 \rightarrow \ldots \rightarrow i_k \rightarrow i_1$ which we write as $(i_1\ldots i_k)$. We are interested in studying when a product of cycles can be written as a product of smaller ones. This is equivalent to decomposing a union of cycles into smaller cycles. For instance, the relation $c_{(123)}c_{(132)}=c_{(12)}c_{(23)}c_{(13)}$ says that the union of these two 3-cycles can be decomposed into three 2-cycles. 

\begin{lem}
\label{cycle-decomp}
Let $\pi_1,\pi_2, \ldots, \pi_m$ be $m \geq 2$ distinct cycles of length $k \geq 3$ with the same support. Then, the product $c_{\pi_1}c_{\pi_2}\cdots c_{\pi_m}$ can be expressed as a product of strictly smaller cycles.
\end{lem}

\begin{proof}
We may assume that all the cycles have support $[k]$. Note that it suffices to prove our lemma for $m=2, 3$.  The following is our {\em key claim}:
given an $l$-cycle $c_{(1si_3\ldots i_l)}$, $l \leq k$, $s \neq 2$, not equal to $c_{(1s(s+1)\ldots k)}$, the product $c_{(1\ldots k)} c_{(1si_3\ldots i_l)}$ can be expressed as a product of cycles of length smaller than $k$. Indeed, suppose that no such expression exists. Let the graphs of $c_{(1\ldots k)}$ and $c_{(1si_3\ldots i_l)}$ be $\mathcal{G}_1$ and $\mathcal{G}_2$ respectively. Color the edges of $\mathcal{G}_1$ red and $\mathcal{G}_2$ blue.
Then $\mathcal{G} = \mathcal{G}_1 \cup \mathcal{G}_2$ contains the cycle $\mathcal{C}_1 = (1s(s+1)\ldots k)$ where the first edge $1 \rightarrow s$ is blue while every other edge is red. Since $s\neq 2$, $\mathcal{C}_1$ has fewer than $k$ vertices. The following algorithm decomposes $\mathcal{G} {\setminus} \mathcal{C}_1$ into cycles:
\begin{enumerate}
\item Initialize $i=1$ and $v_1 = s$.
\item Begin with vertex $v_i$ and take the directed blue path until a vertex $v_{i+1}$ from
the set $\{1, 2, \ldots, v_i-1\}$ is encountered.
\item Take the red path from $v_{i+1}$ to $v_i$. Call the resulting cycle $\mathcal{C}_{i+1}$.
\item If $v_{i+1} = 1$, we are done. Otherwise, increase $i$ by 1 and go to step 2.
\end{enumerate}
Since no decomposition into smaller cycles exists, one of the cycles $\mathcal{C}_i$ has $k$ vertices. In particular, $\mathcal{C}_i$ contains the vertex $1$, so by the above construction, $v_i = 1$. Let $\mathcal{P}$ be the blue path in $\mathcal{C}_i$ from $v_{i-1}$ to $v_i=1$. Since $s$ cannot lie in the interior of the red path $\mathcal{C}_i-\mathcal{P}$ from $1$ to $v_{i-1}$, it must lie on $\mathcal{P}$. The blue edge into $s$ emenates from $1$. However,
 $1$ is the last vertex of $\mathcal{P}$, so $s$ must be the first vertex of $\mathcal{P}$, i.e. $v_{i-1} = s$. This shows that $\mathcal{P}$ is a path from $s$ to $1$ with vertex set $\{s, s+1, \ldots, k, 1\}$. Its union with the blue edge $1 \rightarrow s$ gives a blue cycle contained in $\mathcal{G}_2$, so $\mathcal{G}_2$ equals this cycle. Since $\mathcal{G}_2$ is not the cycle $(1s(s+1) \ldots k)$, it contains an edge $\alpha \rightarrow \beta$ with $s\leq 
\alpha  < \beta \leq k$ and $\beta \neq \alpha +1$. The same argument with $\alpha$ and $\beta$ relabeled as $1$ and $s$ now shows that the vertex set of $\mathcal{G}_2$ in the old labeling is $\{\beta, \beta+1, \ldots, k, 1, \ldots, s, \ldots, \alpha \}$ This is
a contradiction, which proves the key claim.

We return to the lemma. For $m=2$ we simply use the key claim. Suppose $m=3$. Let $\mathcal{G}_1, \mathcal{G}_2$ and $\mathcal{G}_3$ be the three cycles. The $m=2$ case tells us that $\mathcal{G}_1 \cup \mathcal{G}_2$ can be decomposed into smaller cycles $\mathcal{C}_1, \ldots, \mathcal{C}_r$. The trick now is to take the union of some $\mathcal{C}_i$ with $\mathcal{G}_3$ and apply the key claim. If $\mathcal{C}_i$ has at most $|\mathcal{C}_i|-2$ directed edges in common with $\mathcal{G}_3$, we are done. Indeed, we can label the vertices so that $\mathcal{G}_3 = (12\ldots k)$ and $\mathcal{C}_i = (1si_3\ldots i_l)$ with $s \neq 2$. Also, $\mathcal{C}_i \neq (1s(s+1)\ldots k)$, otherwise it has $|\mathcal{C}_i|-1$ edges in common with $\mathcal{G}_3$. Hence, the key claim applies. We are left with the case where each $\mathcal{C}_i$ has $|\mathcal{C}_i|-1$ edges in common with $\mathcal{G}_3$. Assume further that  $\mathcal{C}_1, \ldots, \mathcal{C}_r$ are those constructed by the algorithm in the key claim. It is then not difficult to deduce that either $\mathcal{G}_1 = \mathcal{G}_3$ or $\mathcal{G}_2 = \mathcal{G}_3$. This contradicts the assumption that the three graphs are distinct.
\end{proof}

\begin{prop}
\label{integral}
The algebra $\C[c_*]$ is integral over its subalgebra $\C[P_*]$.
\end{prop}

\begin{proof}
  Let $R_k = \C[P_*, \{c_\pi\}_{|\pi|\leq k}] \subset \C[X]$ be the subring generated by the principal minors and cycles of length at most $k$. Note that $R_n = \C[c_*]$ and $R_2 = \C[P_*]$ since $c_{(ij)} = P_iP_j-P_{ij}$ for all distinct $i,j \in [n]$. Thus, it suffices to show that $R_{k}$ is integral over $R_{k-1}$ for all $3 \leq k \leq n$. In particular, we need to show that each $k$-cycle $c_{\pi}$ is the root of a monic polynomial in $R_{k-1}[z]$ where $z$ is an indeterminate. 

We claim that 
    the monic polynomial $p(z) = \prod_{\pi \in \mathfrak{C}_I} (z-c_\pi)$ is in $R_{k-1}[z]$
    for all $I \subseteq [n]$.
  Indeed, the coefficient of $z^{N-d}$, $1 \leq d \leq N = |\mathfrak{C}_I|$, in $p(z)$ is
\begin{eqnarray*}
\alpha_d \,\,\, = \,\,\,  (-1)^d \! \sum_{\{\pi_1, \ldots, \pi_d\} \subseteq \mathfrak{C}_I} c_{\pi_1}c_{\pi_2}\cdots c_{\pi_d}.
\end{eqnarray*}
Observe that $\alpha_1 = -C_{[k]}$ which  lies in $\C[P_*] \subseteq R_{k-1}$
by Proposition \ref{pm-cs formula}. For $d>1$, we apply Lemma \ref{cycle-decomp} 
which implies that each monomial in $\alpha_d$ can be expressed as a product of smaller cycles. This shows that $\alpha_d \in R_{k-1}$. 
\end{proof}

\begin{cor}
\label{bounded_cycle}
If $\{A_k\}_{k>0}$ is a sequence of complex $n {\times} n$-matrices whose principal minors are bounded, then the cycles $c_\pi(A_k)$ are also bounded.
\end{cor}

\begin{proof}
 Proposition \ref{integral} implies that $c_\pi$ satisfies a monic polynomial with coefficients in $\C[P_*]$. Since the principal minors are bounded, these coefficients are also bounded, so the same is true for $c_\pi$.
\end{proof}

\begin{proof}[Proof of Theorem \ref{thm:closure}]
Suppose $\{A_k\}_{k>0}$ is a sequence of complex $n {\times} n$-matrices whose principal minors tend to $u_* \in \C^{2^n}$. Since the cycle
values $c_\pi(A_k)$ are bounded, we can 
pass to a subsequence and assume that the sequence
of values  for each cycle converges to a complex number
$v_{\pi}$. Lemma \ref{cyclemapclosed} below states that the image
of the cycle map is closed. Hence there exists
an $n {\times} n$-matrix $A$ such that
$c_\pi(A) = v_{\pi}$ for all cycles. The limit minor $u_I$ is expressed in terms of the $v_{\pi}$ using the formula (\ref{eq:pmcs1}). We conclude that the principal minors of the matrix $A$ satisfy
$P_I(A) = u_I$ for all~$I$.
\end{proof}

The following lemma concludes the proof of
Theorem \ref{thm:closure} and this section.

\begin{lem} \label{cyclemapclosed}
Let $M$ be the number of cycles and consider the map
$\gamma :  \C^{n^2} \rightarrow \C^M$
whose coordinates are the cycle monomials $c_\pi$ in $\C[X]$.
Then the image of the monomial map $\gamma$ is a closed subset of $\C^M$.
(So, it is a toric variety).
\end{lem}

\begin{proof}
The general question of when the image of a given monomial map
between affine spaces is closed was studied and answered independently in
 \cite{GMS} and in \cite{KT}. We can apply the characterizations
given in either of these papers to show that the image of
our map $\gamma$ is closed. The key observation is that the cycle
monomials generate the ring of invariants of the 
$(\C^*)^n$-action on $\C[X]$ given by $X \mapsto D \cdot X \cdot D^{-1}$
where $D$ is an invertible diagonal matrix. Equivalently,
the exponent vectors of the monomials $c_\pi$ are the minimal
generators of a subsemigroup of $\N^{n^2}$ that is the solution set of
a system of linear equations on $n^2$ unknowns.
The geometric meaning of this key observation 
is that the monomial map $\gamma$ represents the quotient map
$\, \C^{n^2} \rightarrow \C^{n^2}/ (\C^*)^n$ in the sense
of geometric invariant theory. Now, the results on images
of monomial maps in \cite[\S 3]{GMS} and \cite{KT} ensure that
$\im \gamma$ is closed. 
\end{proof}

\section{Sixty-Five Affine Relations}

We seek to identify generators for the prime ideal $\mathcal{I}_n$ of polynomials in $\C[A_*]$ 
that vanish on the image $\im \phi_a$ of the affine principal minor map. 
Here the $2^n$ coordinates $A_I$ of the vector $A_*$ are regarded as indeterminates.
For $n \leq 3$,  
every vector $u_* \in \C^{2^n}, u_\emptyset = 1,$ is realizable as the principal minors of some $n {\times} n$-matrix, so $\mathcal{I}_n = 0$. In this section, we determine $\mathcal{I}_n$ for the case $n=4$.

Finding relations among the principal minors of a $4 {\times} 4$-matrix is a classical problem posed by MacMahon in 1894 and partially solved by Nanson in 1897 \cite{Ma, M, N}. The relations were discovered by means of ``devertebrated minors'' and trigonometry. Here, we write 
the {\em Nanson relations}
in terms of the cycle-sums. They are the maximal $4 {\times} 4$-minors of the following 
$5 {\times} 4$-matrix:
\begin{eqnarray} \label{5x4matrix}
 \left(
 \begin{array}{cccc}
  C_{123}C_{14} & C_{124}C_{13} & C_{134}C_{12} & 2C_{234}C_{12}C_{13}C_{14}+C_{134}C_{124}C_{123}\\
  C_{124}C_{23} & C_{123}C_{24} & C_{234}C_{21} & 2C_{134}C_{21}C_{23}C_{24}+C_{234}C_{124}C_{123}\\
  C_{134}C_{32} & C_{234}C_{31} & C_{123}C_{34} & 2C_{124}C_{31}C_{32}C_{34}+C_{234}C_{134}C_{123}\\
  C_{234}C_{41} & C_{134}C_{42} & C_{124}C_{43} & 2C_{123}C_{41}C_{42}C_{43}+C_{234}C_{134}C_{124}\\
  1 & 1 & 1 & C_{1234}
 \end{array}
 \right)
\end{eqnarray}
Each of the cycle-sums in this matrix can be rewritten as a 
polynomial in the principal minors $P_I$ using the relations
(\ref{eq:pmcs2}). An explicit example is
$$ C_{123} \,\, = \,\, 2\,A_1 A_2 A_3 
- A_{12} A_3 - A_{13} A_2 -  A_{23} A_1
+ A_{123}. $$
The maximal minors of (\ref{5x4matrix}) give us
five polynomials in the ideal
$\mathcal{I}_4$. Each can be expanded either in terms of
cycle-sums or in terms of principal minors.

Muir \cite{M} and Nanson \cite{N} left open the question 
of whether additional polynomials are needed to generate
the ideal $\mathcal{I}_4$, even up to radical.
We~applied computer algebra methods to answer this question.
In the course of our experimental investigations,
 we discovered the $65$ affine relations that generate the ideal.
 They are the generators of the ideal quotient $\,(\mathcal{K}:g)\,$ where $\mathcal{K}$ is the ideal generated by the five $4 {\times} 4$-minors
  above, and $g$ is the principal 
 $3 {\times} 3$-minor corresponding to the first three rows and columns of  (\ref{5x4matrix}).
Thus the main stepping stone in the proof of Theorem  \ref{thm:affine} is the identity
\begin{equation}
\label{quotientideal}
\mathcal{I}_4 \,\,\, = \,\,\, (\mathcal{K}:g).
\end{equation}
Before proving this, we present a census of the
$65$ ideal generators, and we explain why all
$65$ polynomials are needed and to what extent they are 
uniquely characterized  by the equality of ideals in (\ref{quotientideal}).
The polynomial ring $\C[A_*]$ has $15$ indeterminates that are indexed
by non-empty subsets of $\{1,2,3,4\}$.
It has the following natural multigrading by the group $\Z^4$:
\begin{eqnarray*}
&\deg(A_1) = [1,0,0,0]\,, \, \deg(A_{2}) = [0,1,0,0]\,,\, \ldots \,, \, \deg(A_4) = [0,0,0,1],& \\
&\deg(A_{12}) = [1,1,0,0]\,,\, \ldots\,,\, \deg(A_{234}) = [0,1,1,1]\,,\, \deg(A_{1234}) = [1,1,1,1].& 
\end{eqnarray*}
This $\Z^4$-grading is a positive grading, i.e., each graded component is
a finite-dimensional $\C$-vector space.
Both the ideal $ \mathcal{K}$ and the prime ideal $\mathcal{I}_4$  are homogeneous
in this $\Z^4$-grading. This means that the minimal generators of both ideals
can be chosen to be $\Z^4$-homogeneous, and their number is unique.

Our computation revealed that this number of generators is $65$. Moreover,
we found that the generators lie in $63$ distinct $\Z^4$-graded components.
The component in degree $[5,5,5,5]$ happens to be three-dimensional. We chose
 a $\C$-basis for this $3$-dimensional space of polynomials.
All $62$ other components are one-dimensional,
and these give rise to generators with integer coefficients and content $1$ that are unique up to sign.
The complete census of all $65$ generators is presented in Table \ref{summary}.
For each generator we list its multidegree and its
number of monomials (``size'') in the two expansions,
namely, in terms of cycle-sums $C_I$ and in terms of principal minors $A_I$.

The first four rows of Table \ref{summary}  refer to  the four maximal minors
of the matrix (\ref{5x4matrix}) which involve the last row.
The expansion of any of these four minors in terms of cycle-sums has $32$ monomials and is of total degree $8$. However, the expansion of that polynomial 
in terms of principal minors $A_I$ is much larger: it has  $5234$ monomials.

\begin{table}
\caption{Multidegrees of the $65$ minimal generators of $\mathcal{I}_4$}
\label{summary}
\begin{center}
\begin{tabular}{|ccccc|}
\hline
& \multicolumn{2}{c}{\footnotesize{Cycle-sums}} & \multicolumn{2}{c|}{\footnotesize{Principal Minors}} \\
\footnotesize{No.} & \footnotesize{Size} & \footnotesize{Deg.} & \footnotesize{Size} & \footnotesize{Multidegree} \\
\hline
 1 & 32 & 8 & 5234 & [4, 5, 5, 5] \\
 2 & 32 & 8 & 5234 & [5, 4, 5, 5] \\
 3 & 32 & 8 & 5234 & [5, 5, 4, 5] \\
 4 & 32 & 8 & 5234 & [5, 5, 5, 4] \\
 \hline
 5 & 42 & 9 & 4912 & [4, 4, 6, 6] \\
 6 & 42 & 9 & 4912 & [4, 6, 4, 6] \\
 7 & 42 & 9 & 4912 & [4, 6, 6, 4] \\
 8 & 42 & 9 & 4912 & [6, 4, 4, 6] \\
 9 & 42 & 9 & 4912 & [6, 4, 6, 4] \\
10 & 42 & 9 & 4912 & [6, 6, 4, 4] \\
\hline
11 & 80 & 9 & 5126 & [4, 5, 5, 6] \\
12 & 80 & 9 & 5126 & [4, 5, 6, 5] \\
13 & 80 & 9 & 5126 & [4, 6, 5, 5] \\
14 & 80 & 9 & 5126 & [5, 4, 5, 6] \\
15 & 80 & 9 & 5126 & [5, 4, 6, 5] \\
16 & 80 & 9 & 5126 & [5, 5, 4, 6] \\
17 & 80 & 9 & 5126 & [5, 5, 6, 4] \\
18 & 80 & 9 & 5126 & [5, 6, 4, 5] \\
19 & 80 & 9 & 5126 & [5, 6, 5, 4] \\
20 & 80 & 9 & 5126 & [6, 4, 5, 5] \\
21 & 80 & 9 & 5126 & [6, 5, 4, 5] \\
22 & 80 & 9 & 5126 & [6, 5, 5, 4] \\
\hline
23 & 116 & 9 & 5656 & [5, 5, 5, 5] \\
24 & 116 & 9 & 5656 & [5, 5, 5, 5] \\
25 & 116 & 9 & 5656 & [5, 5, 5, 5] \\
\hline
26 & 91 & 10 & 6088 & [3, 6, 6, 6] \\
27 & 91 & 10 & 6088 & [6, 3, 6, 6] \\
28 & 91 & 10 & 6088 & [6, 6, 3, 6] \\
29 & 91 & 10 & 6088 & [6, 6, 6, 3] \\
\hline
30 & 834 & 11 & 5779 & [5, 5, 5, 7] \\
31 & 834 & 11 & 5779 & [5, 5, 7, 5] \\
32 & 834 & 11 & 5779 & [5, 7, 5, 5] \\
33 & 834 & 11 & 5779 & [7, 5, 5, 5] \\
\hline
\end{tabular}
\begin{tabular}{|ccccc|}
\hline
& \multicolumn{2}{c}{\footnotesize{Cycle-sums}} & \multicolumn{2}{c|}{\footnotesize{Principal Minors}} \\
\footnotesize{No.} & \footnotesize{Size} & \footnotesize{Deg.} & \footnotesize{Size} & \footnotesize{Multidegree} \\
\hline
34 & 163 & 10 & 5234 & [4, 5, 5, 7] \\
35 & 163 & 10 & 5234 & [4, 5, 7, 5] \\
36 & 163 & 10 & 5234 & [4, 7, 5, 5] \\
37 & 163 & 10 & 5234 & [5, 4, 5, 7] \\
38 & 163 & 10 & 5234 & [5, 4, 7, 5] \\
39 & 163 & 10 & 5234 & [5, 5, 4, 7] \\
40 & 163 & 10 & 5234 & [5, 5, 7, 4] \\
41 & 163 & 10 & 5234 & [5, 7, 4, 5] \\
42 & 163 & 10 & 5234 & [5, 7, 5, 4] \\
43 & 163 & 10 & 5234 & [7, 4, 5, 5] \\
44 & 163 & 10 & 5234 & [7, 5, 4, 5] \\
45 & 163 & 10 & 5234 & [7, 5, 5, 4] \\
\hline
46 & 254 & 10 & 5558 & [4, 5, 6, 6] \\
47 & 254 & 10 & 5558 & [4, 6, 5, 6] \\
48 & 254 & 10 & 5558 & [4, 6, 6, 5] \\
49 & 214 & 10 & 6716 & [5, 4, 6, 6] \\
50 & 214 & 10 & 6716 & [5, 6, 4, 6] \\
51 & 214 & 10 & 6716 & [5, 6, 6, 4] \\
52 & 254 & 10 & 5558 & [6, 4, 5, 6] \\
53 & 254 & 10 & 5558 & [6, 4, 6, 5] \\
54 & 254 & 10 & 5558 & [6, 5, 4, 6] \\
55 & 254 & 10 & 5558 & [6, 5, 6, 4] \\
56 & 254 & 10 & 5558 & [6, 6, 4, 5] \\
57 & 254 & 10 & 5558 & [6, 6, 5, 4] \\
\hline
58 & 354 & 10 & 5993 & [5, 5, 5, 6] \\
59 & 354 & 10 & 5993 & [5, 5, 6, 5] \\
60 & 354 & 10 & 5993 & [5, 6, 5, 5] \\
61 & 364 & 10 & 8224 & [6, 5, 5, 5] \\
\hline
62 & 685 & 11 & 5915 & [4, 6, 6, 6] \\
63 & 685 & 11 & 5915 & [6, 4, 6, 6] \\
64 & 685 & 11 & 5915 & [6, 6, 4, 6] \\
65 & 685 & 11 & 5915 & [6, 6, 6, 4] \\
   & & &      & \\ 
\hline
\end{tabular}
\end{center}
\end{table}

\begin{proof}[Proof of Theorem \ref{thm:affine}]
We compute the ideal $(\mathcal{K}:g)$ and find
that it has the $65$ minimal generators above.
We check that each of the five generators of $\mathcal{K}$
vanishes on $\im \phi_a$ but $g$ does not vanish on $\im \phi_a$.
This implies 
$ (\mathcal{K}:g) \subseteq \mathcal{I}_4$.
To prove the reverse inclusion we argue as follows.
Computation of a Gr\"obner basis in terms of cycle-sums reveals
that  $(\mathcal{K}:g)$ is an ideal of codimension $2$,
and we know that this is also the codimension of the prime ideal $\mathcal{I}_4$.
Therefore $\mathcal{I}_4$ is a minimal associated prime of  $(\mathcal{K}:g)$.
To complete the proof, it therefore suffices to show that $(\mathcal{K}:g)$ is a prime ideal.
We do this using the following lemma:

\begin{lem}
\label{stillman}
Let $J \subset k[x_1,x_2, \ldots, x_n]$ be an ideal containing a polynomial $f=gx_1+h$, with $g,h$ not involving $x_1$ and $g$ a non-zero divisor modulo~$J$. Then, $J$ is prime if and only if the elimination ideal
$J \cap k[x_2, \ldots, x_n]$ is prime.
\end{lem}

Lemma \ref{stillman} is due to M.~Stillman and appears in  \cite[Prop. 4.4(b)]{St}.
We apply this lemma to our ideal $J = (\mathcal{K}:g)$ in the
polynomial ring $\C[A_*]$, with $x_1 = A_{1234}$ as the special variable,
and we take the special polynomial $f$ to be the $4 {\times} 4$-minor 
of the matrix  (\ref{5x4matrix}) formed by deleting the fourth row.

We have $\,f \,= \, gA_{1234}+h$ where $g,h$ are polynomials that
do not involve $A_{1234}$. A computation verifies that
$(J:g) = J$, so $g$ is not a zero-divisor modulo $J$.
It remains to show that the elimination ideal $J \cap \C[A_* \backslash A_{1234}]$ is prime.
Now, since $J$ has codimension two, this elimination ideal is principal.
Indeed, its generator is the $4 {\times} 4$-minor of (\ref{5x4matrix})
given by the first four rows. This polynomial has $\Z^4$-degree $[6,6,6,6]$. We  check using computer algebra that it is absolutely irreducible, and
conclude that $J $ is prime.
\end{proof}

\section{A Pinch of Representation Theory}
\label{sec:PRT}

In this section we prove Theorem \ref{thm:projective},
and we explicitly determine the
$14$ polynomials of degree $12$ 
that serve as highest weight vectors for
the  relevant irreducible $G$-modules.
We begin by describing the
general setting for  $n \geq 4$.

Let $\mathcal{V}_n \subset \PP^{2^n-1}$ be the image of the projective principal minor map $\phi$, 
and let $\mathcal{J}_n \in \C[A_*]$ be the homogeneous prime ideal of polynomials 
in $2^n$ unknowns $A_I$ that vanish on $\mathcal{V}_n$.
Clearly, $\mathcal{J}_n$ is invariant under the action of $\mathfrak{S}_n$ on $\C[A_*]$ which comes from permuting the rows and columns of the $n {\times} n$ -matrix. Let ${\rm GL}_2(\C)$ denote the group of invertible complex $2 {\times} 2$-matrices, and consider the vector space $V = V_1 \otimes V_2 \otimes \cdots \otimes V_n$ where each $V_i \simeq \C^2$ with basis $e^{i}_0, e^{i}_1$. We identify
a basis vector $e^{1}_{j_1} \otimes e^{2}_{j_2} \otimes \cdots \otimes e^{n}_{j_n}$ of
the tensor product $V$ with the unknown $A_I$ where $i \in I$ if and only if $j_i = 1$, for all $i$. 

The natural action of the $n$-fold product $G = {\rm GL}_2(\C) {\times} \cdots {\times} {\rm GL}_2(\C)$
 on $V \simeq \C^{2^n}$ extends to its coordinate ring $\C[A_*]$. This action commutes with the 
 map $\,\phi\,$ in (\ref{projectivephi}). Here $G$ acts on the
parameter space of  $n {\times} 2n$-matrices $(A,B)$
by having its $i$-th factor ${\rm GL}_2(\C)$ act on the
$n {\times} 2$-matrix formed by the $i$-th columns of $A$ and $B$.
This argument, due to J. M. Landsberg,
shows that  the prime ideal $\mathcal{J}_n$ is invariant under $G$.
See also \cite[\S 6]{HS} and \cite[Thm.~1.1]{O}.

\begin{cor}
The set $\mathcal{V}_n $ is closed in $ \PP^{2^n-1}$, {\it i.e.}~it is a projective
variety.
\end{cor}

\begin{proof}
For any index set $I\subseteq [n]$, there exists a group element $g \in G$ which takes $A_\emptyset$ to $A_I$. Thus, every affine piece $\,U_I = \{u_* \in \mathcal{V}_n: u_I \neq 0\}\,$ 
of the constructible set $\,\mathcal{V}_n = \im \phi\,$ is isomorphic to $\,U_\emptyset = \im \phi_a$.
The latter image is closed by Theorem \ref{thm:closure}. Therefore,
$\mathcal{V}_n$ is an irreducible projective variety. 
\end{proof}

The action of  the Lie group $G$ gives rise to an action of the  Lie algebra $\mathfrak{g} = \mathfrak{gl}_2(\C) {\times} \cdots {\times} \mathfrak{gl}_2(\C)$  on the polynomial ring $ \C[A_*]$ by differential operators. Here $\mathfrak{gl}_2(\C)$ 
denotes the ring of complex $2 {\times} 2$-matrices. Indeed, the vector
$$ ({\bf 0}, \ldots, {\bf 0},
\begin{pmatrix}w & x\\y & z\end{pmatrix}, {\bf 0}, \ldots, {\bf 0}) \quad \in \quad \mathfrak{g} ,$$
whose entries are zero matrices of format $2 {\times} 2$ except in the $i$-th coordinate, acts
on the polynomial ring $\C[A_*]$ as the linear differential operator
\begin{eqnarray*}
\sum_{i \notin I} \left( w A_{I}\frac{\partial}{\partial A_{I}}
+ x A_{I\cup \{i\}}\frac{\partial}{\partial A_{I}}
+ y A_{I}\frac{\partial}{\partial A_{I\cup \{i\}}}
+ z A_{I\cup \{i\}}\frac{\partial}{\partial A_{I\cup \{i\}}} \right).
\end{eqnarray*}
Here the sum is over all $I \subseteq [n]$ not containing $i$. We extend this action to all of $\mathfrak{g}$ by linearity. If all the coordinate matrices of an element $h \in \mathfrak{g}$ are 
strictly upper triangular, we say that $h$ is a {\em raising operator}. Similarly, $h$ is a
{\em lowering operator} if all the matrices are strictly lower triangular.

We now focus on the case $n=4$. Let $\mathcal{I}^h$ be the ideal generated by the homogenizations of the 65 affine generators in Theorem \ref{thm:affine} with respect to $A_\emptyset$.
This is a subideal of the homogeneous prime ideal $\mathcal{J} = \mathcal{J}_4$
we are interested in. Our strategy is to identify a suitable intermediate ideal
between $\mathcal{I}^h \subset \mathcal{J}$.

\begin{proof}[Proof of Theorem \ref{thm:projective}]
Let  $\,\mathcal{K} = G\mathcal{I}^h\,$ be the ideal generated by
the image of $\mathcal{I}^h$ under the group $G$.
Then $\mathcal{K}$ is a $G$-invariant ideal of $\C[A_*]$
that is contained in the prime ideal $\mathcal{J}$.
Since $G$ acts transitively on the affine charts  $U_I$,
and since $\mathcal{I}^h$ coincides with the unknown ideal $\mathcal{J}$
on the chart $U_\emptyset$,
we conclude that the ideal $\mathcal{K}$ defines
our projective variety   $\mathcal{V}_4$ scheme-theoretically.

By definition, the ideal $\mathcal{K}$ is generated by its degree-$12$ component $\mathcal{K}_{12}$.
To prove Theorem \ref{thm:projective} we need to show that
$\mathcal{K}_{12}$ is a $G$-module of $\C$-dimension $718$,
and that it decomposes into $14$ irreducible $G$-modules. 
As a $G$-module, the graded component $\mathcal{K}_{12}$ is generated by 
the homogenizations of the $65$ polynomials
 in Table \ref{summary}.  Representation theory as  in \cite{LM, O}
 tells us that the unique highest weight vectors of the irreducible $G$-modules contained
 in $\mathcal{K}_{12}$ can be found by applying raising operators $h \in \mathfrak{g}$ to these $65$
 generators.
 
    Indeed, consider the 1st, 5th and 26th polynomials in Table \ref{summary}. When written in terms of cycle-sums, these polynomials have $32$, $42$ and $91$ terms respectively. Let $D$, $E$ and $F$ be their homogenizations with respect to $A_\emptyset$. By applying
     the raising and lowering operators in the Lie algebra $\mathfrak{g}$,
     one  checks that all $65$ generators lie in the $G {\rtimes} \mathfrak{S}_4$-orbit of $D$, $E$ and $F$.
     Thus,     
$$
\mathcal{K}_{12} \,\,\, = \,\,\, M_D \oplus M_E \oplus M_F
$$
where $M_D$, $M_E$ and $M_F$ are the
$G$-modules spanned by the $G {\rtimes} \mathfrak{S}_4$-orbits of the polynomials
$D$, $E$ and $F$ respectively. Furthermore, as in \cite{LM}, we write
$$ S_{ijkl} \, = \,
S_{(12-i,i)}(\C^2) \otimes S_{(12-j,j)}(\C^2) \otimes S_{(12-k,k)}(\C^2) \otimes S_{(12-l,l)}(\C^2)
$$
for the tensor product of Schur powers of $\C^2$.
The dimension of  $S_{ijkl}$ equals $  (13-2 i) (13-2 j) (13-2 k) (13-2 l)$.
Our three $G$-modules decompose as
\begin{eqnarray*}
&&M_D \,\, \simeq \,\, S_{4555} \oplus S_{5455} \oplus S_{5545} \oplus S_{5554}, \\
&&M_E \,\, \simeq \,\, S_{4466} \oplus S_{4646} \oplus S_{4664} \oplus S_{6446} \oplus S_{6464} \oplus S_{6644}, \\
&& M_F \,\, \simeq \,\, S_{3666} \oplus S_{6366} \oplus S_{6636} \oplus S_{6663}.
\end{eqnarray*}
The above three vector spaces have dimensions 540, 150 and 28 respectively.
This shows that $\,\dim(\mathcal{K}_{12})=718$, and the proof
of Theorem \ref{thm:projective} is complete.
\end{proof}

Many questions about relations among
principal minors remain open at this point, even for $n=4$.
The most basic question is whether the ideal $\mathcal{K}$
is prime, that is, whether $\mathcal{K} = \mathcal{J}$ holds (cf. Remark \ref{rem:BR}).
Next, it would be desirable to find a nice determinantal representation 
for the polynomials $E$ and $F$, in analogy to 
$D$ being the homogenization of the determinant of the last four rows in (\ref{5x4matrix}).
We know little about the prime ideal $\mathcal{I}_n$ for $n \geq 5$.
It contains various natural images of the ideal $\mathcal{I}_4$, 
but we do not know whether these generate. The most optimistic
conjecture would state that $\mathcal{I}_n$
is generated by the ${\rm GL}_2(\C)^n$-orbit of the polynomials $D$, $E$ and $F$.
The work of Oeding \cite{O} gives hope that at least the
set-theoretic version might be within reach:

\begin{conj}
The variety  $\mathcal{V}_n \subset \PP^{2^n-1}$ is cut out by equations
 of degree~$12$.
\end{conj}

\section{Singularities of the Hyperdeterminant}

Our object of study is the projective variety $\mathcal{V}_4$
which is parametrized by the principal minors of a generic $4 {\times} 4$-matrix. We have seen that
$\mathcal{V}_4$ is a variety of codimension two in the projective space
$\PP^{15}$. That ambient space is the projectivization of the
vector space $\C^2 \otimes \C^2 \otimes \C^2 \otimes \C^2$
of $2 {\times} 2 {\times} 2 {\times} 2 $-tables $ a = (a_{ijkl})$.
This section offers a geometric characterization of
the variety $\mathcal{V}_4$.

The articles \cite{HS, O} show that the variety parametrized
by the principal minors of a  symmetric matrix is closely 
related to the $2 {\times} 2 {\times} 2$-hyperdeterminant. It is 
thus quite natural for us to ask whether such a relationship
also exists in the non--symmetric case. We shall argue that this
is indeed the case.

The {\em hyperdeterminant} of format  $2 {\times} 2 {\times} 2 {\times} 2 $
is a homogeneous polynomial of degree $24$ in $16$ unknowns 
having $2,894,276 $ terms  \cite{HSYY}.  Its expansion into cycle-sums
was found to have $13,819$ terms. The hypersurface $\nabla$ 
of this hyperdeterminant
 consists
of all tables $\,a \in \PP^{15} \,$ such that the hypersurface 
$$\,\mathcal{H}_a \,\,\,= \,\,\, \bigl\{ (x,y,z,w) \in
\PP^1 {\times} \PP^1 {\times} \PP^1 {\times} \PP^1\,:
\,\sum_{i=0}^1\sum_{j=0}^1\sum_{k=0}^1 \sum_{l=0}^1
a_{ijkl} x_i y_j z_k w_l\, = 0 \bigr\}$$
 has a singular point.
Weyman and Zelevinsky \cite{WZ} showed that
the hyperdeterminantal hypersurface  $\nabla \subset \PP^{15}$ is singular
in codimension one. More precisely, by
\cite[Thm. 0.5 (5)]{WZ}, the singular locus $\nabla_{\rm sing}$
of $\nabla$ is the union in $\PP^{15}$ of eight
irreducible projective varieties, each having dimension $13$:
\begin{equation}
\label{nablasing}
 \nabla_{\rm sing} \,\,\,\, = \,\,\,\,
\nabla_{\rm node}(\emptyset) \,\,\,\cup \,
\bigcup_{1 \leq i < j \leq 4} \! \nabla_{\rm node}(\{i,j\}) 
 \,\,\,\cup  \,\,\,
\nabla_{\rm cusp} .
\end{equation}
Here, the {\em node component} $\nabla_{\rm node}(\emptyset)$ 
is the closure of the set of tables $a$ such that
$\mathcal{H}_a$ has two singular points
$(x,y,z,w)$ and $(x',y',z',w')$ with
$x \not= x'$, $y \not= y'$, $z \not= z'$ and $w \not= w'$.
The {\em extraneous component} $\nabla_{\rm node}(\{1,2\})$
is the closure of the set of tables $a$ such that
$\mathcal{H}_a$ has two singular points $(x,y,z,w)$ and $(x,y,z',w')$,
and similarly for the other five extraneous components.
Finally, the cusp component  $\nabla_{\rm cusp} $ parametrizes
all tables $a$ for which $\mathcal{H}_a$ has a  triple point.
The connection to our study is given by the following result:

\begin{thm} \label{thm:hyperdet}
The node component in the singular locus of the
$ \, 2 {\times} 2 {\times} 2 {\times} 2 \,$-
hyperdeterminant coincides with the variety  parametrized by the
principal minors of a generic $4 {\times} 4$-matrix.
In symbols, we have $\, \mathcal{V}_4 \, = \,\nabla_{\rm node}(\emptyset)$.
\end{thm}

\begin{proof}
We now dehomogenize by setting 
$x_0=y_0=z_0=w_0=1$,
$x_1=x, \,  y_1=y,\,
 z_1=z\,$ and $\, w_1=w$.
Let $(c_{ij})$ be a generic 
complex $4 {\times} 4$-matrix
and consider the ideal in 
$\C[x,y,z,w]$ generated by
the $3 {\times} 3$-minors~of 
\begin{equation}
\label{fourbyfour}
\begin{pmatrix}
c_{11}+x  &  c_{12}  &  c_{13}  &  c_{14}  \\
c_{21}  &  c_{22}+y  &  c_{23}  &  c_{24}  \\
c_{31}  &  c_{32}  &  c_{33}+z  &  c_{34}  \\
c_{41}  &  c_{42}  &  c_{43}  &  c_{44}+w 
\end{pmatrix}.
\end{equation}
We claim that the variety of this ideal consists of two distinct
points in $\C^4$. We prove this by a computation.
Regarding the
$c_{ij}$ as unknowns, we compute a
Gr\"obner basis for the ideal of $3 {\times} 3$-minors
with respect to the lexicographic term order
$\, x > y > z > w \,$ over the base field
$K = \Q(c_{11}, c_{12}, \cdots, c_{44})$.
The output shows that the ideal is radical and 
the initial ideal equals
$\langle x,y,z,w^2 \rangle $.

The determinant of (\ref{fourbyfour})
is the affine multilinear form
$$ F (x,y,z,w) \,\, = \,\, \sum_{i=0}^1   \sum_{j=0}^1   \sum_{k=0}^1 
 \sum_{l=0}^1  a_{ijkl} \cdot x^i y^j z^k w^l $$
whose coefficients are the principal minors of the
$4 {\times} 4$-matrix $(c_{ij})$. The claim 
established in the previous paragraph
implies that the system of equations
$$ F \,=\,
\frac{\partial F}{\partial x} \,=\,
\frac{\partial F}{\partial y} \,=\,
\frac{\partial F}{\partial z} \,=\,
\frac{\partial F}{\partial w} \,=\, 0  $$
has two distinct solutions over the algebraic closure of
$K = \Q(c_{11},  \cdots, c_{44})$.
These two solutions correspond to two distinct
singular points of the hypersurface  $\,\mathcal{H}_a \,$ in $\,
\PP^1 {\times} \PP^1 {\times} \PP^1 {\times} \PP^1$.
From this we conclude that the table $\,a = (a_{ijkl})\,$
of principal minors of $(c_{ij})$ lies in the
node component
$\,\nabla_{\rm node}(\emptyset)$.

Our argument establishes the inclusion 
$ \mathcal{V}_4  \subseteq \nabla_{\rm node}(\emptyset)$.
Both sides are irreducible subvarieties of $\PP^{15}$,
and in fact, they share the same dimension, namely $13$.
This means they must be equal.
\end{proof}

Our results in Section \ref{sec:PRT} give
an explicit description of the equations that  define the first 
component in the decomposition 
(\ref{nablasing}).
This raises the problem of identifying the equations of
the other seven components.

\smallskip

\begin{remark} \label{rem:BR}
After posting the first version of this paper on the {\tt arXiv}, we 
learned that some of the results in this paper have already been addressed in \cite[\S 4]{BR}.
Specifically, Landsberg's observation that $\mathcal{V}_n$ is $G$-invariant
coincides with \cite[Theorem 4.2]{BR}, and our Theorem \ref{thm:hyperdet}
coincides with \cite[Theorem 4.6]{BR}. Coincidentally, without reference to dimensions, we proved $\mathcal{V}_4  \subseteq \nabla_{\rm node}(\emptyset)$ directly while \cite[Theorem 4.6]{BR} gives the other inclusion. Moreover, at the very end of the paper \cite{BR}, it is stated that
{\em ``... the variety has degree $28$, with ideal generated by
a whopping $718$ degree $12$ polynomials''}.
This appears to prove our conjecture (at the end of Section 4)
that the ideal $\mathcal{K} $ is actually prime. We verified that
the ideal $\mathcal{K}$ has degree $28$, but we did not
yet succeed in verifying that $\mathcal{K}$ is saturated
with respect to the irrelevant maximal ideal.
We tried to do this computation by specializing the $16$ unknowns to linear 
in fewer unknowns but this leads to an ideal which is not prime.
It thus appears that the ideal $\mathcal{K}$ 
is not Cohen-Macaulay. 
\end{remark}

\bigskip \smallskip

\noindent {\bf Acknowledgements.} Shaowei Lin was supported by
graduate fellowship from A*STAR (Agency for Science, Technology and Research, Singapore).
Bernd Sturmfels was supported in part by the U.S.~National Science Foundation 
(DMS-0456960 and DMS-0757236). We thank Luke Oeding for conversations and software
that helped us in Section \ref{sec:PRT}, we thank Harald Helfgott for bringing the classical papers \cite{M,N,St} to our attention, and we thank Eric Rains for pointing us to
his earlier results in \cite[\S 4]{BR}, discussed above.

\bigskip 
\bigskip
\bigskip

\noindent {\bf Authors' addresses:}

\bigskip

\noindent Shaowei Lin and Bernd Sturmfels, Dept.~of Mathematics,
University of California, Berkeley, CA 94720, USA,
\url{{shaowei,bernd}@math.berkeley.edu}

\end{document}